\documentclass[runningheads,a4paper]{llncs}

\usepackage{amssymb}
\usepackage{graphicx}
\usepackage{enumerate}
\usepackage{url}
\usepackage{epstopdf}
\urldef{\mailsa}\path|{j.m.portegies,g.r.sanguinetti,s.p.l.meesters,r.duits}@tue.nl|
\newcommand{\keywords}[1]{\par\addvspace\baselineskip
\noindent\keywordname\enspace\ignorespaces#1}
\newcommand{\bx}{\mathbf{x}}
\newcommand{\bn}{\mathbf{n}}
\newcommand{\bR}{\mathbf{R}}
\newcommand{\OmegaBF}{\mbox{\boldmath$\Omega$}}
\newcommand{\by}{\mathbf{y}}
\newcommand{\bc}{\mathbf{c}}

\newcommand{\tW}{\tilde{W}}
\newcommand{\be}{\mathbf{e}}
\newcommand{\cA}{\mathcal{A}}
\newcommand{\tp}{\tilde{p}}
\newcommand{\tU}{\tilde{U}}
\newcommand{\mA}{\mathcal{A}}
\newcommand{\quotpo}{\mathbb{R}^3 \rtimes S^2}
\newcommand{\tPhi}{\tilde{\Phi}}
\newcommand{\mL}{\mathcal{L}}
\newcommand{\tK}{\tilde{K}}
\newcommand{\desda}{\Leftrightarrow}
\newcommand{\ul}{\mathbf}
\newcommand{\R}{\mathbb{R}}

\begin{document}

\mainmatter  % start of an individual contribution

% first the title is needed
\title{New Approximation of a Scale Space Kernel on SE(3) and Applications in Neuroimaging}
% a short form should be given in case it is too long for the running head
\titlerunning{New Approximation of a Scale Space Kernel on SE(3)}

% the name(s) of the author(s) follow(s) next
%
% NB: Chinese authors should write their first names(s) in front of
% their surnames. This ensures that the names appear correctly in
% the running heads and the author index.
%
\author{Jorg Portegies$^1$%
\and Gonzalo Sanguinetti$^1$\and Stephan Meesters$^{1,2}$ \and \\Remco Duits$^{1,3}$ \thanks{\small
The research leading to the results of this article has received funding from the European Research Council under the EC’s 7th Framework Programme (FP7/2007– 2014)/ERC grant agreement No. 335555.}
}
\authorrunning{New Approximation of a Scale Space Kernel on SE(3)}
% (feature abused for this document to repeat the title also on left hand pages)
 %\footnote{Joint main authors.}} %\author{ANONYMOUS}%\footnote{Joint main authors.}}

% the affiliations are given next; don't give your e-mail address
% unless you accept that it will be published
\institute{$^1$Dept. of Mathematics and Computer Science, Eindhoven University of Technology, \\ $^2$Academic Center for Epileptology \& Maastricht UMC+, Heeze \\ $^3$Dept. of Biomedical Engineering, Eindhoven University of Technology, \\
\mailsa\\
}

%
% NB: a more complex sample for affiliations and the mapping to the
% corresponding authors can be found in the file "llncs.dem"
% (search for the string "\mainmatter" where a contribution starts).
% "llncs.dem" accompanies the document class "llncs.cls".
%

\toctitle{Lecture Notes in Computer Science}
\tocauthor{Authors' Instructions}
\maketitle

\begin{abstract}
We provide a new, analytic kernel for scale space filtering of dMRI data. The kernel is an approximation for the Green's function of a hypo-elliptic diffusion on the 3D rigid body motion group SE(3), for fiber enhancement in dMRI. The enhancements are described by linear scale space PDEs in the coupled space of positions and orientations embedded in SE(3). As initial condition for the evolution we use either a Fiber Orientation Distribution (FOD) or an Orientation Density Function (ODF). Explicit formulas for the exact kernel do not exist. Although approximations well-suited for fast implementation have been proposed in literature, they lack important symmetries of the exact kernel. We introduce techniques to include these symmetries in approximations based on the logarithm on SE(3), resulting in an improved kernel. Regarding neuroimaging applications, we apply our enhancement kernel (a) to improve dMRI tractography results and (b) to quantify coherence of obtained streamline bundles.
\end{abstract}

\section{Introduction}

In dMRI it is assumed that axons grouped in bundles in human brain tissue restrict the Brownian motion of water molecules such that more diffusion occurs along the bundles \cite{leBihan1986}. By measuring the decay of signal due to diffusion in many directions it is possible to obtain information about the underlying microstructure of the brain tissue and further processing of this data provides clues about the anatomical brain connectivity. After a pre-processing procedure in which the raw data is corrected for e.g. distortions and motion artefacts, different models can be used for further processing of the data \cite{Descoteaux2014}. We construct from the data a fiber orientation distribution (FOD) function on positions and orientations, representing the probability density of finding a fiber in a certain position and orientation. For this we use Constrained Spherical Deconvolution (CSD) \cite{Tournier2007}, but the methods in this paper can be combined with any model that outputs an FOD or an Orientation Density Function (ODF) of water molecules. 

\begin{figure}[t!]
  \centering
  \includegraphics[width=0.85\textwidth]{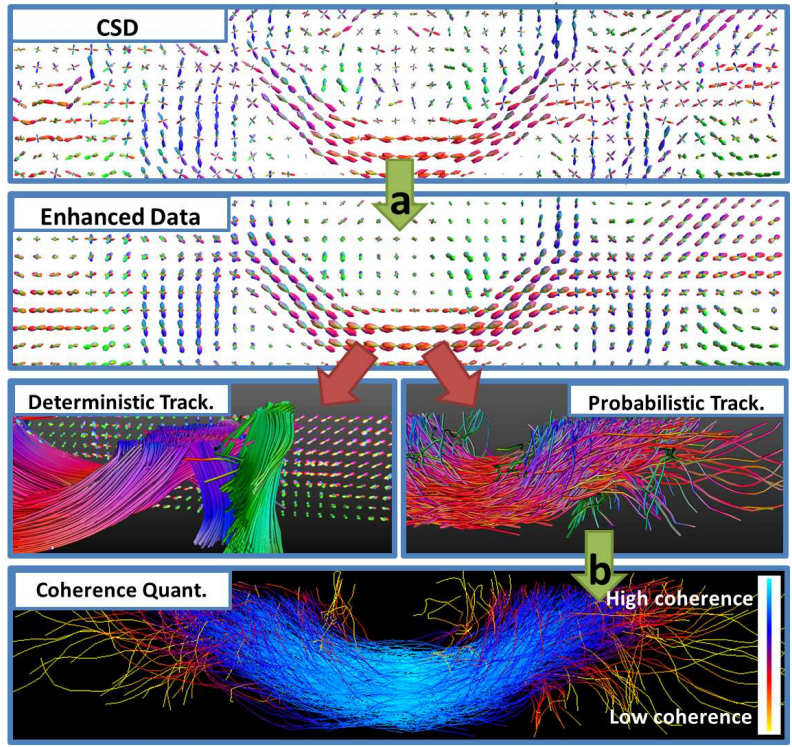}
\caption{The pipeline used in this paper for processing dMRI data. CSD provides an FOD, that can be enhanced by convolution with our proposed kernel as in Section \ref{se:applications}\textbf{a}. We visualize the FOD $U$ by a field of \emph{glyphs}, where at each position $\by$ on a grid and each orientation $\bn$ the radius is of the glyph is directly proportional to $U(\by,\bn)$. The color corresponds to direction of the orientation. Then, deterministic or probabilistic tractography produces fiber bundles, that often contain deviating fibers. They can be classified with our proposed coherence quantification, see Section \ref{se:applications}\textbf{b}.}\label{fig:pipeline}
\end{figure}

For regularization of dMRI data, various methods exist that include contextual information on position and/or orientation space \cite{tschumperle2002,Burgeth2009a,Duits2011,Reisert2011,schultz2012towards,Momayyez2013,Duits2013,becker2014,batard2014}. In this paper we pursue the method of scale spaces on the group of positions and rotations $SE(3)$. For this we consider diffusions described by the Fokker-Planck equations of hypo-elliptic Brownian motion \cite[Sect. 4.2]{Duits2011}. The effect of these evolutions on the FOD is that elongated structures are enhanced while crossings of bundles are preserved. No explicit formulas exist for the exact Green's function of the PDE, but approximations exist in literature \cite{Duits2011}: one is the product of two SE(2)-kernels and one is based on the logarithmic modulus on $SE(3)$. However, two important invariance properties of the exact kernel are not automatically obeyed in these approximations, as pointed out in Section \ref{se:kernelapprox}. 

The first (and main) contribution of this paper is that we provide a new analytic kernel approximation for Brownian motion on $SE(3)$ that is, in contrast to previous analytic approximations, well-defined on the quotient. As such, it carries the appropriate symmetry. Furthermore, the novel, more precise approximation is still well-suited for fast kernel implementations \cite{rodrigues2010}. 

The second contribution in this work is the application of the improved approximations in two different places (\textbf{a}) and (\textbf{b}) in the dMRI pipeline as in Fig.~\ref{fig:pipeline}. We provide two clinically relevant experiments to illustrate this. We apply the enhancements to the FOD and show the advantages in the tractography result \textbf{(a)}. Secondly, we use the enhancement kernel to construct a density on the space of positions and orientations, based on a set of fibers \textbf{(b)}. From this we derive a measure for the coherence of fibers within a fiber bundle. The symmetry included in the new kernel leads to a reduction in computation time of the measure.

Section \ref{se:methods} covers the theory of the paper and compares previous and new approximation kernels. In Section \ref{se:applications} the applications \textbf{(a)} and \textbf{(b)} are presented.

\section{Approximations of Scale Spaces on $SE(3)$}\label{se:methods}
First we explain in Section \ref{se:embeddingse3} how the space of 3D positions and orientations, on which FODs are defined, can be embedded in the group of 3D rigid body motions $SE(3)$. This is followed by a brief introduction to scale spaces on $SE(3)$, see Section \ref{se:scalespaces}, together with a discussion of the two required invariances. In Section \ref{se:kernelapprox} we give two known kernel approximations and we propose an adaptation for the logarithmic approximation, such that the desired invariances are induced.

\subsection{The Embedding of $\mathbb{R}^3 \rtimes S^2$ into $SE(3)$}\label{se:embeddingse3}
Contextual enhancement of a function $U : \mathbb{R}^3 \times S^2 \rightarrow \mathbb{R}^+$, representing an FOD, means improving the alignment of elongated structures present in the FOD. Such alignment can only be done in a space where positions and orientations are coupled. Therefore we embed $\mathbb{R}^3 \times S^2$ in the group of 3D rigid body motions $SE(3)$, with group product $(\bx,\bR)(\bx',\bR') = (\bx + \bR \bx', \bR \bR')$, where $\bR, \bR' \in SO(3)$, the 3D rotation group, and $\bx,\bx'\in \mathbb{R}^3$. Square integrable functions $U$ relate to a specific set of functions $\tilde{U}:SE(3)\rightarrow \mathbb{R}^+$ via
\begin{equation}
U \leftrightarrow \tilde{U}
\desda
\left\{
\begin{array}{l}
U(\by,\bn) = \tU(\by,\bR_{\bn}) \textrm{ and } \tU(\by,\mathbf{R})=U(\by,\mathbf{R}\mathbf{e}_{z})
\end{array}
\right\},
\end{equation}
where from now on, $\bR_{\bn}$ denotes \emph{any} rotation such that $\bR_{\bn} \be_z = \bn$, and where $\ul{e}_{z}=(0,0,1)^T$ denotes our reference axis.
The functions $\tilde{U}$ carry a symmetry, as they are right-invariant w.r.t. right-action
$\ul{R} \mapsto \ul{R} \bR_{\be_z,\alpha}$, where from now on $\bR_{\ul{a},\psi}$ denotes the counter-clockwise rotation around axis $\ul{a}$ by angle $\psi$.
This implies we should consider left cosets on $SE(3)$, with equivalence relation
\begin{equation}
\begin{array}{l}
g:=(\by,\bR) \equiv g':=(\by',\bR') \desda g^{-1} g' \in H \desda
\by = \by' \textrm{ and }\bR' = \bR \bR_{\be_z,\alpha},
\end{array}
\end{equation}
for some $ \alpha \in [0,2\pi)$, where from now on $H$ denotes the subgroup of $SE(3)$ whose elements are equal to $h=(\ul{0},R_{\ul{e}_{z},\alpha})$.
The total set of functions $\tilde{U}$, mentioned earlier, is given by:
\[
\mathbb{L}_{2}^{R}(SE(3)):= \{\tilde{U} \in \mathbb{L}_{2}(SE(3))\;|\; \tilde{U}(gh)=\tU(g) \textrm{ for all }g \in SE(3), h \in H\}.
\]
From now on, we identify the function $U:\mathbb{R}^3 \times S^2 \rightarrow \mathbb{R}$ with a function on the group quotient $\mathbb{R}^3 \rtimes S^2 := SE(3)/H = SE(3) / (\{\mathbf{0}\}\times SO(2))$.  Throughout the paper we use a tilde to distinguish functions on the group from functions on the quotient. Next, we present scale spaces for such functions $\tU$ and $U$.

\subsection{Scale Spaces and Group Convolution on $\mathbb{R}^3 \rtimes S^2$ and $SE(3)$}\label{se:scalespaces}
We define the enhancement evolutions on $SE(3)$ in terms of the left-invariant vector fields. They can be considered as differential operators on locally defined smooth functions \cite{aubin2001}. Left-invariant vector fields are obtained from a Lie-algebra basis for the tangent space $T_e(SE(3))$ at unity element $e = (0,I) \in SE(3)$, say $\{A_1,\dots,A_6\}$. This is done with the pushforward $(L_g)_*$ of the left-multiplication $L_g q = gq$: $\mA_g \tilde{U} = \mA_e(\tilde{U} \circ L_g) = ((L_g)_* \mA_e)(\tilde{U})$, for all smooth $\tilde{U}: SE(3) \rightarrow \mathbb{R}$. Explicit formulas for the vector fields $g \mapsto \left.\mA_i \right|_g$ can be obtained by 
\begin{equation}\label{eq:Ai}
\mA_j|_g \tilde{U} = \lim_{t\rightarrow0} \frac{\tilde{U}(g\, e^{tA_j}) - \tU(g)}{t}, \qquad j = 1,\dots,6,
\end{equation}
where $T_e(SE(3)) \ni A \mapsto e^{tA} \in SE(3)$ denotes the exponential map yielding the 1-parameter subgroup $\{e^{tA} \; \vline \; t\in \mathbb{R}\}$. Note that $\mA_i|_e = A_i$. We write $\{\mA_i\}_{i=1}^6$, where $\{\mathcal{A}_1,\mathcal{A}_2,\mathcal{A}_3\}$ are spatial and $\{\mathcal{A}_{4},\mathcal{A}_{5},\mathcal{A}_6\}$ are rotation vector fields, for which we use the explicit expressions in two different Euler angle parametrizations as given in \cite{Duits2011}. For geometric intuition, see Fig.~\ref{fig:intuition}.
\begin{figure}[t!]
  \centering
  \includegraphics[width=0.95\textwidth]{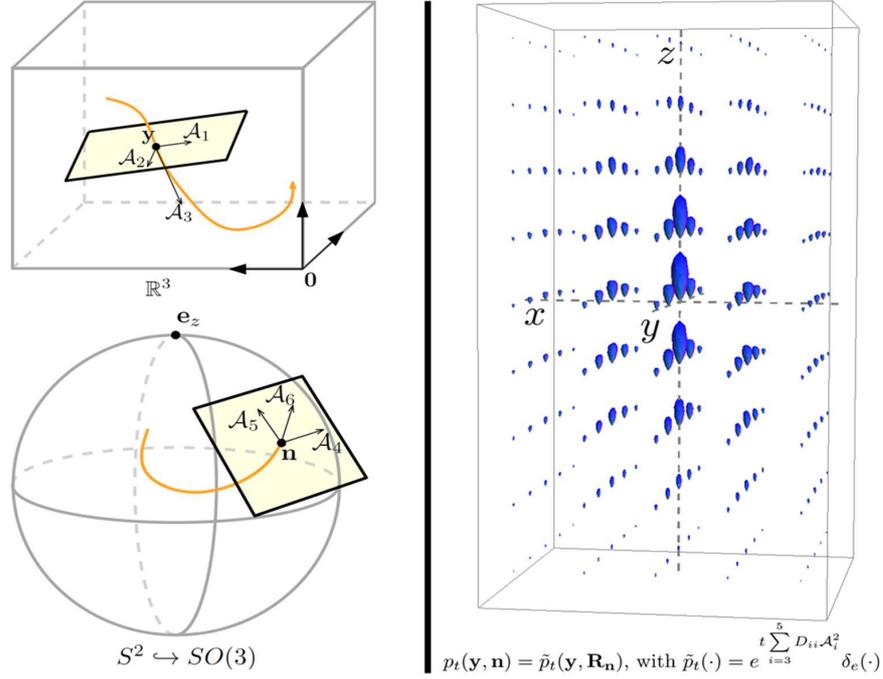}
\caption{\emph{Left:} left-invariant vector fields $\{\mA_i\}_{i=1}^6$ in $SE(3)$, restricted to $(\by,\bR_{\bn})$. \emph{Right:} glyph visualization of the impulse response $p_t$ as the scale space kernel of (\ref{eq:CEquotient}), with $D_{33} = 1$, $D_{44} = D_{55}=0.02$, $t = 2$.}\label{fig:intuition}
\end{figure} 
Now the (hypo-elliptic) diffusion process describing Brownian motion on $SE(3)$ \cite{Duits2011}, for $\tU$ on $SE(3)$ can be formulated as:
\begin{equation}\label{eq:CESE(3)}
\left\{ \begin{array}{rl}
 \frac{\partial }{\partial t}\tW(\by,\bR,t) &= D_{33} (\cA_3)^2 \tW(\by,\bR,t) + D_{44} (\cA_4^2 + \cA_5^2) \tW(\by,\bR,t),\\[4pt]
 \tW(\by,\bR,0) &= \tU(\by,\bR).
\end{array} \right.
\end{equation}
This yields a scale space representation $\tW$ of the function $\tU$ with scale parameter $t>0$. Parameters $D_{33}>0$ and $D_{44}>0$ influence the amount of spatial and angular diffusion, respectively. Solutions for $\tW$ can be found via finite difference approximations of the PDE \cite{Creusen2011}, or via convolution $\tU$ with a kernel. Fast kernel implementations exist \cite{rodrigues2010} and are particularly suitable for our applications. The $SE(3)$-convolution with a probability kernel $\tp_t:SE(3) \rightarrow \mathbb{R}^+$ is given by
\begin{equation}
\tW(g,t) = (\tp_t \ast_{SE(3)} \tU)(g) = \int_{SE(3)} \tp_t(q^{-1}g) \tU(q) {\rm d}q.
\end{equation}
where ${\rm d}q={\rm d}\ul{x}{\rm d}R$, representing the Haar measure, for all $q=(\ul{x},R) \in SE(3)$.
This evolution is well-defined also on $\quotpo$ and can be written as
\begin{equation}\label{eq:CEquotient}
\left\{ \begin{array}{l}
 \frac{\partial}{\partial t}W(\by,\bn,t) = D_{33} (\ul{n} \cdot \nabla)^2 W(\by,\bn,t) + D_{44} \Delta_{LB} W(\by,\bn,t), \\
 W(\by,\bn,0) = U(\by,\bn),
\end{array} \right.
\end{equation}
with $\Delta_{LB}$ the Laplace-Beltrami operator on $S^2$. Solutions of (\ref{eq:CEquotient}) are found by $\mathbb{R}^{3}\rtimes S^{2}$-convolution with the exact solution kernel $p_t:\mathbb{R}^{3}\rtimes S^{2} \to \mathbb{R}^{+}$:
\begin{equation}\label{eq:shifttwistconv}
\begin{array}{ll}
  W(\by,\bn,t) &= \left(p_t \ast_{\mathbb{R}^3 \rtimes S^2} U  \right)(\by,\bn)\\
  &= \int \limits_{S^2} \int \limits_{\mathbb{R}^3} p_t (\bR_{\bn'}^T(\by - \by'), \bR_{\bn'}^T \bn) \cdot U(\by', \bn') {\rm d} \by' {\rm d} \sigma(\bn').
  \end{array}
\end{equation}
The kernels $\tilde{p}_t$ and $p_t$ should satisfy certain symmetries, as shown in the following definition, lemma and corollary.
\begin{definition}\label{def:legaloperator}
An operator $\tilde{\Phi}:\mathbb{L}_2^R(SE(3)) \rightarrow \mathbb{L}_2^R(SE(3))$ is called \emph{legal} if %it satisfies
\[%begin{equation*}
\left\{ \begin{array}{ll} \mathcal{L}_g \circ \tilde{\Phi} = \tilde{\Phi} \circ \mathcal{L}_g & \qquad \forall g \in SE(3), \\
\tilde{\Phi} \circ R_h = \tilde{\Phi} = R_{h'} \circ \tilde{\Phi} & \qquad \forall h,h' \in H,
\end{array} \right.
\]
with $\mathcal{L}_g$ and $R_h$ the left- and right-regular action on SE(3) respectively, see \cite{Duits2011}.
\end{definition}
Legal operators $\tilde{\Phi}$ induce well-posed operators $\Phi$ on $\mathbb{L}_{2}(\mathbb{R}^{3} \rtimes S^{2})$ via
$\Phi(U)(\ul{y},\ul{n})= \tilde{\Phi} (\tilde{U})(\ul{y}, \bR_{\ul{n}})$. In particular, this holds for legal scale space operators.
\begin{lemma}\label{le:lemma}
Let $\tilde{\Phi}:\mathbb{L}_2^R(SE(3)) \rightarrow \mathbb{L}_2^R(SE(3))$ be linear and legal, and assume it maps $\mathbb{L}_2(SE(3))$ into $\mathbb{L}_{\infty}(SE(3))$. Then we have:
\begin{enumerate}
\item identity: $(\tilde{\Phi} (\tU))(g) = \int \limits_{SE(3)} \tilde{K}(g,q)\,\tU(q)\, {\rm d} q = (\check{\tp} \ast_{SE(3)} \tU)(g),
\textrm{ with } \\ \tilde{K}(g,q) = \tilde{K}(e,(g^{-1}q)) =:\check{\tp}(q^{-1}g)$ and with $\tilde{p}(g):=\tilde{K}(g,e)=\check{\tilde{p}}(g^{-1})$.\mbox{}\vspace{0.2cm}\mbox{}
\item symmetry: $
\tilde{K}(gh, qh') = \tilde{K}(g,q)\textrm{ for all } h, h' \in H$ and all $g \in SE(3)$.
%\end{equation}
\item preservation of correspondences:
$(U \leftrightarrow \tU) \Rightarrow \Phi(U) \leftrightarrow \tilde{\Phi}(\tU)$ with
\begin{equation}\label{int}
\begin{array}{lll}
(\Phi (U))(\by,\bn) &= \int \limits_{\mathbb{R}^3 \times S^2} k(\by,\bn,\by',\bn') U(\by',\bn') {\rm d} \by'{\rm d} \sigma(\bn'), \qquad \textrm{ with } \\
k(\by,\bn,\by',\bn') &= \frac{1}{2\pi}\tilde{K}((\by,\bR_{\bn}),(\by',\bR_{\bn'})) \\
 &= \frac{1}{2\pi}\tp((\by',\bR_{\bn}')^{-1}(\by,\bR_{\bn})) = p(\bR_{\bn'}^{-1}(\by-\by'),\bR_{\bn'}^{-1}\bn).
\end{array}
\end{equation} 
\end{enumerate}
\end{lemma}
\begin{proof}
By the Dunford-Pettis Theorem cf.~\cite{arendt1994}, $\tPhi$ is a kernel operator and $\|\tPhi \|^2=\sup \limits_{q \in SE(3)} \int \limits_{SE(3)}
|\tilde{K}(g,q)|^2{\rm d}q$, so we can write %\begin{equation}
$(\tPhi(\tU))(g) = \int \limits_{SE(3)} \tK(g,q) \tU(q) {\rm d} q$. \\
1. Operator $\tPhi$ is legal, so from the first identity in Definition \ref{def:legaloperator} we have
\begin{equation}\label{eq:leftinv}
\forall_{g' \in SE(3)}\forall_{\tU \in \mathbb{L}_{2}^{R}(SE(3))}\;:\;
(\tPhi \circ \mL_g \circ \tU)(g') = (\mL_g \circ \tPhi \circ \tU)(g').
\end{equation}
This holds for all $\tU$ and by writing (\ref{eq:leftinv}) in integral form, one can check that
\begin{equation}
\begin{array}{l}
\tK(g^{-1}g',q) = \tK(g',g q),  \
\tK(e,g^{-1}q) = \tK(g,q), \textrm{ for all }g,g',q \in SE(3).
\end{array}
\end{equation}
2. The second identity in Definition \ref{def:legaloperator}, together with the unimodularity of the group $SE(3)$ gives $\tK(g, q h^{-1}) = \tK(g,q) = \tK(g h',q)$, so $\tK$ must be right-invariant under subgroup $H$ with respect to both entries.\\
3. Suppose $U \leftrightarrow \tU$, then $\tilde{U}(\by',R_{\bn'})=U(\by',\ul{n}')$ and
\begin{equation}
\begin{array}{ll}
(\tPhi(\tU)) (\by,\bR_{\ul{n}}) &= \int \limits_{SE(3)} \tK(\by,\bR_{\ul{n}},\by',\bR')\tU(\by',\bR') {\rm d} \by'{\rm d} \bR' \\
 &= 2\pi \int \limits_{\quotpo} \tK(\by,\bR_{\ul{n}} ,\by',\bR_{\bn'})\, \tilde{U}(\by',R_{\bn'})\, {\rm d} \by' {\rm d} \sigma(\bn').
 \end{array}
\end{equation}

Now $\tilde{p}(g)=\tilde{K}(g,e)$, $\check{\tilde{p}}(g)=\tilde{K}(e,g)$. Left-invariance then implies (\ref{int}). $\hfill \Box$
\end{proof}
\begin{corollary}\label{corr:1}
The exact scale space kernels $\tilde{p}_{t}:SE(3) \to \R^{+}$ and $p_{t}:\R^{3}\rtimes S^{2} \to \R^+$ satisfy the following symmetries
\begin{equation}\label{eq:symm1}
p_{t}(\ul{y},\ul{n})= p_{t}(\bR_{\ul{e}_{z},\alpha}\ul{y},\bR_{\ul{e}_{z},\alpha}\ul{n}) \textrm{ and } \tilde{p}_{t}(h^{-1}g)=\tilde{p}_{t}(g)=\tilde{p}_{t}(gh')
\end{equation}
for all $t>0$, $\alpha>0$, $(\ul{y},\ul{n}) \in \R^{3}\rtimes S^{2}$, $g \in SE(3)$, $h,h' \in H$. Moreover, 
\begin{equation}\label{eq:symm2}
 p_t(\by,\bn) = p_t(-\bR_{\bn}^T \by, \bR_{\bn}^T \be_z) \textrm{ and } \tp_t(g) = \tp_t(g^{-1}).
\end{equation}
\end{corollary}
\begin{proof}
The symmetry (\ref{eq:symm1}) is due to Lemma \ref{le:lemma}. The second symmetry (\ref{eq:symm2}) is due to reflectional invariance $(\mA_3,\mA_4,\mA_5) \mapsto (-\mA_3,-\mA_4,-\mA_5)$ in the diffusion (\ref{eq:CESE(3)}) and reflection on the Lie algebra corresponds to inversion on the group. $\hfill \Box$
\end{proof}
\begin{remark}\label{remark:symmetry}
Note that (\ref{eq:symm2}) in terms of $k$ would be: $k(\by,\bn,\mathbf{0},\be_z) = k(\mathbf{0},\be_z,\by,\bn)$, by the relation $k(\by,\bn,\by',\bn') = p(\bR_{\bn'}^{T}(\by-\by'),\bR_{\bn'}^{T}\bn)$. This means that $k$ defines a symmetric measure: evaluation in $(\by,\bn)$ of a kernel centered around the unity element $(\mathbf{0},\mathbf{e}_z)$ should be equal to evaluation in the unity element of a kernel centered around $(\by,\bn)$.
\end{remark}
Now the Gaussian kernel approximation, based on the logarithm on $SE(3)$ and the theory of coercive weighted operators on Lie groups \cite{terElst1998}, presented in earlier work \cite[ch:8,thm.11]{Duits2013} does not satisfy this symmetry. Next we will improve it by a new practical analytic approximation which does satisfy the property.

\subsection{New vs. Previous Kernel Approximations}\label{se:kernelapprox}
We first present two existing approximations for $p_t$ and $\tp_t$, that do not automatically carry over the properties we have shown in the previous section. A possible approximation kernel for $p_t$ is based on a direct product of two $SE(2)$-kernels, see \cite[Eq.(10), Eq.(11)]{rodrigues2010}, to which we refer as $p_t^{prev,1}$. This approximation is easy to use since it is defined in terms of the Euler angles of the corresponding orientations. However, in Fig.~\ref{fig:symmanalysis} we show that the symmetries described before are not preserved by $p_t^{prev,1}$ and errors tend to be larger when $D_{44}$ and $t$ increase. Therefore we move to an approximation for the $SE(3)$-kernel $\tp_t$, for which we show that it can be adapted such that it has the important symmetries. We need the logarithm on $SE(3)$ for this approximation, so first consider an exponential curve in $SE(3)$, given by $\tilde{\gamma}(t)= g_0 \exp \left(t \sum \limits_{i=1}^{6} c^{i}A_i\right)$. The logarithm $\log_{SE(3)}: SE(3) \to T_{e}(SE(3))$ is bijective, and it is given by
\begin{equation}
\log_{SE(3)}(g)=\log_{SE(3)}\left(\exp\left(\sum \limits_{i=1}^{6}c^i(g)A_i\right)\right) =\sum \limits_{i=1}^{6}c^i(g)A_i,
\end{equation}
and we can relate to this the vector of coefficients $\bc(g)=(\bc^{(1)}(g),\bc^{(2)}(g))=(c^{1}(g),\cdots, c^{6}(g))^{T}$. 
We define matrix $\OmegaBF$ as follows:
\begin{equation}
\OmegaBF := \left(\begin{array}{ccc}
0 & -c^6 & c^5 \\
c^6 & 0 & -c^4 \\
-c^5 & c^4 & 0
\end{array}
\right), \qquad \bR = e^{t \OmegaBF}, \qquad \bR \in SO(3).
\end{equation}
We can write $\bR$ in terms of Euler angles, $\bR = \bR_{\be_z,\gamma} \bR_{\be_y,\beta} \bR_{\be_z,\alpha}$. Let the matrix $\OmegaBF_{\gamma,\beta,\alpha}$ be such that
$\exp(\OmegaBF_{\gamma,\beta,\alpha})=\bR_{\be_z,\gamma} \bR_{\be_y,\beta} \bR_{\be_z,\alpha}$. The spatial coefficients $\bc^{(1)} = \bc^{(1)}_{\bx,\gamma,\beta,\alpha}$ are given by the following equation:
\begin{equation}\label{eq:c1}
\bc^{(1)}  = \left(I - \frac12 \OmegaBF_{\gamma,\beta,\alpha} + q^{-2}_{\gamma,\beta,\alpha} \left(1 - \frac{q_{\gamma,\beta,\alpha}}{2} \cot \left(\frac{q_{\gamma,\beta,\alpha}}{2} \right) \right)(\OmegaBF_{\gamma,\beta,\alpha})^2 \right) \bx,
\end{equation}
where $q_{\gamma,\beta,\alpha}$ is the (Euclidean) norm of $\bc^{(2)}$. Then another approximation for the kernel $\tp_t$ is given by \cite{Duits2011}:
\begin{equation}
  \tp^{log}_t(\bc(g)) = \frac{1}{(4\pi t^2 D_{33} D_{44})^2} e^{-\frac{|\log_{SE(3)}(g)|^2_{D_{33},D_{44}}}{4t}},
\end{equation}
  with the smoothed variant of the weighted modulus, \cite[Eq.78,79]{Duits2013}, given by
\begin{equation}
    |\log_{SE(3)}(\cdot)|_{D_{33},D_{44}} = \sqrt[4]{\frac{|c^1|^2+|c^2|^2}{D_{33}D_{44}}+\frac{|c^6|^2}{D_{44}}+\left(\frac{(c^3)^2}{D_{33}}+\frac{|c^4|^2 + |c^5|^2}{D_{44}}\right)^2}.
\end{equation}
Now the difficulty lies in the fact that the function $U$ is defined on the quotient for elements $(\ul{y}, \bR_{\bn})$, where the choice for $\alpha$ in the rotation matrix $\bR_{\bn}$ (mapping $\ul{e}_{z}$ onto $\ul{n} \in S^2$) is not of importance. However, the logarithm is only well-defined on the group $SE(3)$, not on the quotient $\R^{3}\rtimes S^{2}$, and explicitly depends on the choice of $\alpha$. It is therefore not straightforward to use this approximation kernel such that the invariance properties in Corollary~\ref{corr:1} are preserved. In view of Corollary~\ref{corr:1}, we need both left-invariance and right-invariance for $\tilde{p}_t(g)$ under the action of the subgroup $H$. As right-invariance is naturally included via $\tilde{p}_{t}(\ul{y},R)=p_{t}(\ul{y},R \ul{e}_{z})$, left-invariance is equivalent to inversion invariance. In previous work the choice of $\alpha=0$ is taken, giving rise to the approximation
\begin{equation}
p_t^{prev,2}(\by,\bn(\beta,\gamma)) = \tp_t^{log}(\bc(\by,\bR_{\be_z,\gamma}\bR_{\be_y,\beta})).
\end{equation}
However, this section is not invariant under inversion, as is pointed out in Fig.~\ref{fig:symmanalysis}. In contrast, we propose to take the section $\alpha=-\gamma$, which is invariant under inversion (since 
$(\ul{R}_{\ul{e}_{z},\gamma}\ul{R}_{\ul{e}_{y},\beta}\ul{R}_{\ul{e}_{z},-\gamma})^{-1}=
 \ul{R}_{\ul{e}_{z},\gamma}\ul{R}_{\ul{e}_{y},-\beta}\ul{R}_{\ul{e}_{z},-\gamma}$).
Moreover, this choice for estimating the kernels in the group is natural, as it provides the
weakest upper bound kernel since by direct computation one has \mbox{$\alpha=-\gamma \Rightarrow c_{6}=0$}.
Finally, this choice indeed provides us the correct symmetry for the Gaussian approximation of $p_t(\ul{y},\ul{n})$ as stated in the following theorem.

\begin{figure}[t!]
  \centering
  \includegraphics[width=\textwidth]{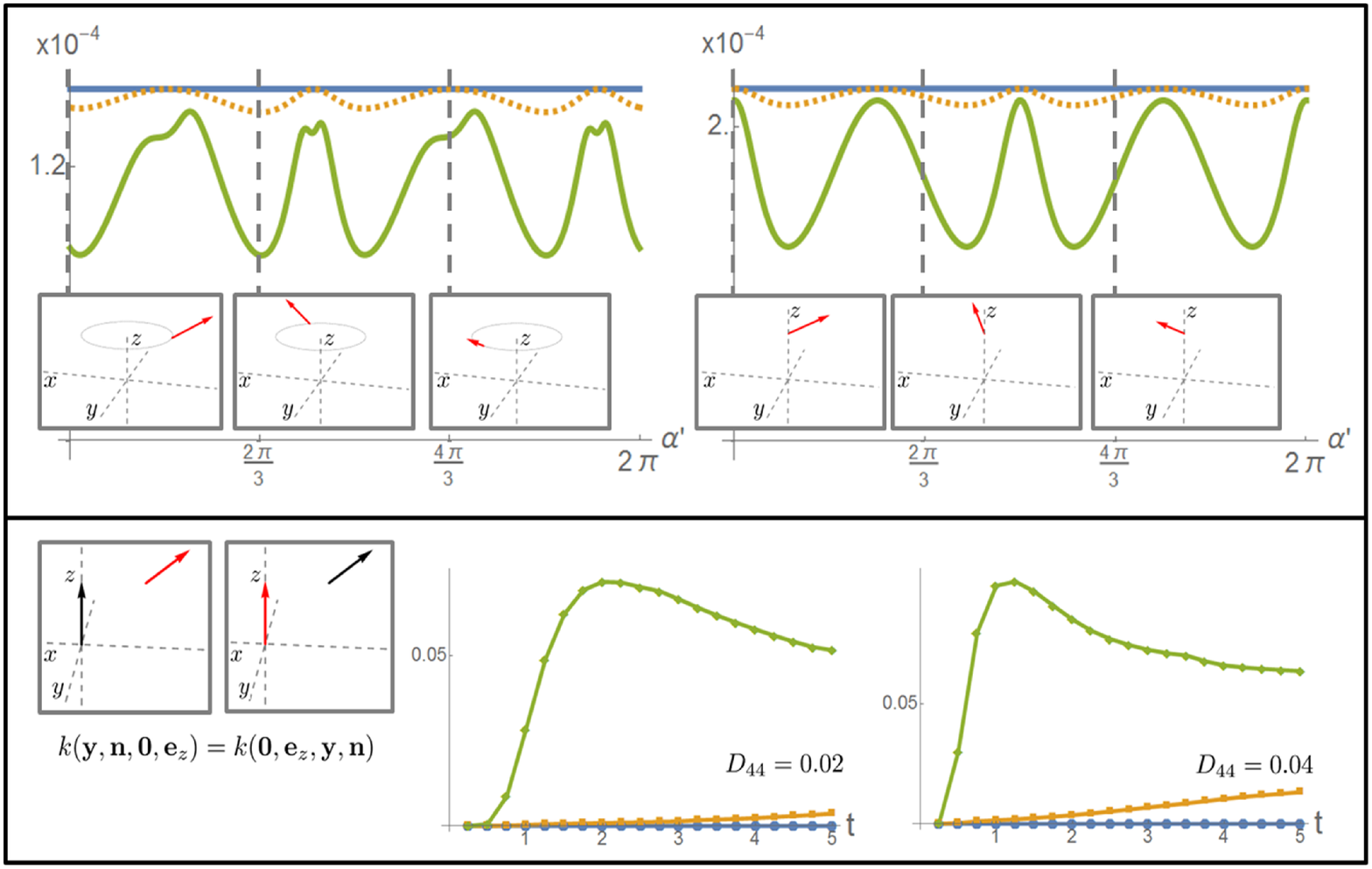}
  \caption{\emph{Top}: 2 positions and orientations, rotated around the $z$-axis. The graphs display $p_t(\bR_{\be_z,\alpha'}\by,\bR_{\be_z,\alpha'}\bn)$, for $p^{prev,1}_t$ in green, $p_t^{prev,2}$ in orange and $p_t^{new}$ in blue. Parameters $D_{33} = 1, D_{44}=0.02, t = 4$. \emph{Bottom}: $k(\by,\bn,\mathbf{0},\be_z)$ should be equal to $k(\mathbf{0},\be_z,\by,\bn)$, see Remark \ref{remark:symmetry}. The absolute sum of the difference of the two evaluations on a $5\times5\times5$ grid with at each point 42 uniformly distributed orientations is shown for $k^{prev,1}$, $k^{prev,2}$ and $k^{new}$, corresponding to $p^{prev,1}_t$, $p_t^{prev,2}$ and $p_t^{new}$, respectively, with coloring as above. Parameter $D_{33}= 1$ and $t, D_{44}$ are as shown.} 
\label{fig:symmanalysis}
\end{figure}

\begin{theorem}
When the approximate kernel $p_t^{new}$ on the quotient is related to the approximate kernel on the group $\tp_t^{log}$ by
\begin{equation}\label{eq:defappkernel}
p_t^{new}(\by,\bn(\beta,\gamma)) := \tp_t^{log}(\bc(\by,\bR_{\be_z,\gamma}\bR_{\be_y,\beta}\bR_{\be_z,-\gamma})),
\end{equation}
i.e. we make the choice $\alpha = -\gamma$, we have the desired $\alpha$-left-invariant property
\begin{equation}
p_t^{new}(\by,\bn) = p_t^{new}(\bR_{\be_z,\alpha'}\by, \bR_{\be_z,\alpha'}\bn), \qquad \alpha' \in [0,2\pi].
\end{equation}
and the symmetry property
\begin{equation}\label{symm}
p_t^{new}(\by,\bn) = p_t^{new}(-\bR^T_{\bn}\by,\bR_{\bn}^T\be_{z})
\end{equation}
\end{theorem}
\begin{proof}
We start by proving the $\alpha$-invariance. Following definition (\ref{eq:defappkernel}) we have
\[
p_t^{new}(\bR_{\be_z,\alpha'}\by, \bR_{\be_z,\alpha'}\bn(\beta,\gamma)) = \tp_t^{log}(\bc(\bR_{\be_z,\alpha'}\by, \bR_{\be_z,\gamma+\alpha'}\bR_{\be_y,\beta}\bR_{\be_z,-(\gamma + \alpha')}))
\]
Recall that the matrix $\OmegaBF_{\gamma,\beta,\alpha}$ is defined by $\bR_{\gamma,\beta,\alpha} := \bR_{\be_z,\gamma}\bR_{\be_y,\beta}\bR_{\be_z,\alpha}= e^{\OmegaBF_{\gamma,\beta,\alpha}}$. Therefore in our case, where we choose $\alpha = -\gamma$, we find for all $\alpha'$:
\[
\begin{array}{lr}
\bR_{\be_y,\beta} &= \bR_{\be_z,\gamma + \alpha'}^{-1} e^{\OmegaBF_{\gamma + \alpha',\beta,-(\gamma + \alpha')}}\bR_{\be_z,\gamma + \alpha'} = e^{\bR_{\be_z,\gamma + \alpha'}^{-1} (\OmegaBF_{\gamma+\alpha',\beta,-(\gamma+\alpha')}) \bR_{\be_z,\gamma+\alpha'}} \\&= e^{\bR_{\be_z,\gamma }^{-1} \bR_{\be_z,\alpha'}^{-1} (\OmegaBF_{\gamma+\alpha',\beta,-(\gamma+\alpha')}) \bR_{\be_z,\alpha'} \bR_{\be_z,\gamma}}
\end{array}
\]
We see that $\bR_{\be_z,\alpha'} (\OmegaBF_{\gamma,\beta,-\gamma})\bR_{\be_z,\alpha'}^{-1}  = \OmegaBF_{\gamma+\alpha',\beta,-(\gamma+\alpha')}$. From this we deduce:
\begin{equation}\label{eq:relationc2}
\begin{array}{l}
\bc^{(2)}_{\gamma+\alpha',\beta,-(\gamma+\alpha')} = \bR_{\be_z,\alpha'} \bc^{(2)}_{\gamma,\beta,-\gamma}, \textrm{ and } %\\
q_{\gamma+\alpha',\beta,-(\gamma+\alpha')} = q_{\gamma,\beta,-\gamma} =:q,
\end{array}
\end{equation}
and together with (\ref{eq:c1}) it gives $\bc^{(1)}_{\bR_{\be_z,\alpha'}\bx,\gamma+\alpha',\beta,-(\gamma+\alpha')} = \bR_{\be_z,\alpha'} \bc^{(1)}_{\bx,\gamma,\beta,-\gamma}$. Combining this with (\ref{eq:relationc2}) gives
\begin{equation}
\bc_{\bR_{\be_z,\alpha'}\bx,\gamma+\alpha',\beta,-(\gamma+\alpha')} = Z^T_{\alpha'} \bc_{\bx,\gamma,\beta,-\gamma}, \textrm{ with } Z_{\alpha'} = \left(\begin{array}{cc} \bR_{\be_z,\alpha'}^T & 0 \\ 0 & \bR_{\be_z,\alpha'}^T \end{array}\right)
\end{equation}
It follows immediately that $c_1^2 + c_2^2$, $c_3$, $c_4^2 + c_5^2$ and $c_6$ are independent of $\alpha'$. The proof for $\alpha$-invariance is completed by stating that given $\alpha' \in [0,2\pi]$:
\[
\begin{array}{l}
p_t^{new}(\bR_{\be_z,\alpha'}\by, \bR_{\be_z,\alpha'}\bn(\beta,\gamma)) = \tp_t^{log}(\log_{SE(3)}(\bR_{\be_z,\alpha'}\by,\ul{R}_{\gamma+\alpha',\beta,-(\gamma+\alpha')}))=\\
\tp_t^{log}(Z_{\alpha'}^T\log_{SE(3)}(\by,\ul{R}_{\gamma,\beta,-\gamma}))= \tp_t^{log}(\log_{SE(3)}(\by,\ul{R}_{\gamma,\beta,-\gamma}))= p_t^{new}(\by,\bn(\beta,\gamma)).
\end{array}
\]
The symmetry property (\ref{symm}) directly follows from the fact that
\[g^{-1}=\exp_{SE(3)}\left(-\sum \limits_{i=1}^{6}c^{i} A_i\right) =
(\ul{y},\ul{R})^{-1}=(-\ul{R}^{-1}\ul{y},\ul{R}^{-1}),
\] 
the fact that our weighted modulus
on $T_{e}(SE(3))$
is invariant under reflection $\ul{c}\mapsto (-\ul{c})$, and the fact that the section $\alpha=-\gamma$ is invariant under inversion. 
$\hfill \Box$ \\
\end{proof}

\section{Neuroimaging Applications}\label{se:applications}
The newly proposed enhancement kernel is used within the pipeline depicted in Fig.~\ref{fig:pipeline} for further processing of the dMRI data. There are two places where this is useful: for enhancement of the FOD obtained with e.g. CSD, and for quantifying the coherence of a fiber within a bundle. 

\textbf{a. Enhancing FODs:}
To illustrate the use of enhancement of FODs, we use the artifical IEEE ISBI 2013 Reconstruction Challenge dataset \cite{Daducci2013}. CSD is applied to a simulated dMRI signal, yielding an FOD as in Fig.~\ref{fig:effect_enhancement}. The tractography was done with MRtrix \cite{Tournier2012}, with seed points randomly chosen in the bundle. Only the fibers that pass through the indicated spheres are shown. The alignment of glyphs improves due to the enhancements and this results in a better tracking: bundles are reconstructed fuller and less fibers take a wrong exit from the bundles. 
\begin{figure}[t!]
  \centering
  \includegraphics[width=\textwidth]{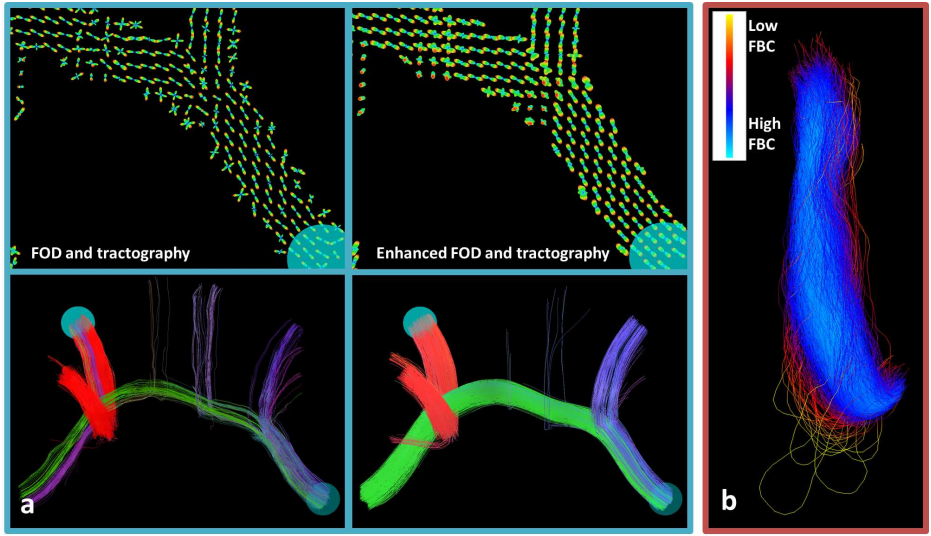}
\caption{\emph{Blue frame:} Glyph field visualization, recall Fig. \ref{fig:pipeline}, of the FOD obtained with CSD from an artificial dataset (top left) can be improved significantly with enhancements (top right). Especially alignment on the boundary of the bundle is improved. This results in a more plausible tractography result (bottom). Here, $D_{33}=1$, $D_{44}=0.02$, $t=4$. \emph{Red frame:} example of the coherence quantification of a fiber bundle, the Optic Radiation.}\label{fig:effect_enhancement}
\end{figure}

\textbf{b. Coherence quantification of tractographies:}
The result of a probabilistic tractography is typically less well-structured than in the deterministic case, recall Fig.~\ref{fig:pipeline}. We employ our new approximation kernel $p_t^{new}$ to construct an $\quotpo$ scale space representation of the set of fibers from a tractography. This density is locally low for a fiber that (partly) deviates in position or orientation from the other fibers in the bundle, as we explain next.

Suppose we have a set of $N$ streamlines (fibers), the $i^{th}$ fiber having $N_i$ points $\by_i^j \in \mathbb{R}^3$ and we set the orientation $\bn_i^j = (\by_{i+1}^j - \by_i^j)/||\by_{i+1}^j-\by_i^j ||\in S^2$. Next we use the set of streamline points $\left\{ (\by_i^j,\pm \bn_i^j) \; | \; j = 1,\dots N_i, \; i=1,\dots,N\right\}$ as follows in the initial condition of our scale space evolution (\ref{eq:CEquotient}):

\begin{equation}
U(\by,\bn) = \frac{1}{N_{tot}} \sum_{\sigma = 1}^2 \sum_{i=1}^N \sum_{j=1}^{N_i} \delta_{(\by_i^j,(-1)^{\sigma}\bn_i^j)}(\by,\bn),
\end{equation} 
with $N_{tot}$ the total number of fiber points and  $\delta_{(\by,\bn)}$ a $\delta$-distribution in $\quotpo$, centered around $(\by,\bn)$. The solution is given by the convolution $W(\cdot,t) = (p_t^{new} \ast_{\quotpo} U)(\cdot)$. Intuitively, letting this sum of $\delta$-distributions diffuse according to (\ref{eq:CEquotient}) results in a density indicating for each position and orientation how well it is aligned, in the coupled $\quotpo$-sense, with the fiber bundle. 

In practice, we evaluate the convolution $(p_t^{new} \ast_{\quotpo} U)$ only in the fiber points, giving $N_{tot}^2$ kernel evaluations. However, this is reduced to $N_{tot}^2/2$ evaluations, since the order of arguments in $k^{new}$ is irrelevant for every pair of fiber points, thanks to the included symmetry in the kernel. This implies a reduction of the computation time. Summation over one fiber then gives a value, which is the final measure for the fiber to bundle coherence (FBC). Fibers with low FBC can be removed from the tractography result and coherent fibers remain. A local FBC can be computed by summing over parts of the fibers instead of the entire fiber. Fig.~\ref{fig:effect_enhancement} shows this local FBC for a tractography of the Optic Radiation, a brain white matter bundle connecting the Lateral Geniculate Nucleus and the visual cortex. It can be seen that where fibers deviate from the bundle, the FBC is lower and hence they could be excluded.

\section{Conclusion}
We have introduced a new approximation for the kernel of a linear scale space PDE (\ref{eq:CEquotient}) on $\quotpo$. This is done by the embedding of $\quotpo$ into $SE(3)$, on which the PDE coincides with the Fokker-Planck equation for a Brownian motion process. Due to this embedding, the kernel is subject to two constraints, recall Corollary \ref{corr:1}. We have shown the application of the kernel in two dMRI neuroimaging applications. First, we have used the kernel implementation in combination with CSD to enhance fiber structures in dMRI and thereby improve tractography results. Secondly, the kernel was used to construct a measure for coherence of fibers in a fiber bundle. Thanks to the appropriate symmetries of the scale space kernels, this measure is based on symmetric distances on $\quotpo$ and a reduction of computation time is obtained. In future work we aim to quantify the improvement in tractography due to the enhancements, and we pursue the use of the FBC for better construction of the structural connectome \cite{rodrigues2013}.

\end{document}